\documentclass[a4paper,draft,reqno,12pt]{amsart}
\usepackage[english]{babel}
\usepackage{amsmath}
\usepackage{amssymb}
\usepackage{amscd}
\usepackage{amsthm}
\usepackage{euscript}
\usepackage[all]{xy}
\usepackage{longtable}
\newtheorem{proposition}{Proposition}

\newtheorem{lemma}{Lemma}
\newtheorem{theorem}{Theorem}

\newtheorem{corollary}{Corollary}
\theoremstyle{definition}

\newtheorem{example}{Example}
\theoremstyle{remark}
\newtheorem {remark}{Remark}

\DeclareMathOperator{\Spec}{Spec}

\DeclareMathOperator{\Aut}{Aut}

\DeclareMathOperator{\Cl}{Cl}
\DeclareMathOperator{\GL}{GL}

\DeclareMathOperator{\codim}{codim}

\DeclareMathOperator{\rk}{rk}

\def\GG{{\mathbb G}}
\def\FF{{\mathbb F}}
\def\CC{{\mathbb C}}
\def\KK{{\mathbb K}}

\def\PP{{\mathbb P}}
\def\AA{{\mathbb A}}
\def\XX{{\mathbb X}}

\def\OOO{\mathcal{O}}

\def\mmm{\mathfrak{m}}

\begin{document}
\date{}
\title[Prehomogeneous modules of commutative algebraic groups]{Prehomogeneous modules of commutative linear algebraic groups}
\author{Ivan Arzhantsev}
\thanks{The research was supported by the grant RSF-DFG 16-41-01013}
\address{National Research University Higher School of Economics, Faculty of Computer Science, Kochnovskiy Proezd 3, Moscow, 125319 Russia}
\email{arjantsev@hse.ru}



\maketitle

\begin{abstract}
Let $A$ be a finite dimensional commutative associative algebra with unit over an algebraically closed field of characteristic zero. The group $G(A)$ of invertible elements is open in $A$ and thus $A$ has a structure of a prehomogeneous $G(A)$-module. We show that every prehomogeneous module of a commutative linear algebraic group appears this way. In particular, the number of equivalence classes of prehomogeneous $G$-modules is finite if and only if the corank of $G$ is at most $5$.
\end{abstract}

\section{Introduction}
\label{sec1}

Let $G$ be a linear algebraic group over an algebraically closed field $\KK$ of characteristic zero. A finite dimensional $G$-module $V$ is called prehomogeneous if the linear action $G\times V\to V$ is effective and has an open orbit in $V$. Prehomogeneous modules play an important role in geometry, number theory and analysis, as well as representation theory.

For connected simple algebraic groups  prehomogeneous modules were classified by Vinberg~\cite{Vi}. A classification of irreducible prehomogenenous modules for connected reductive algebraic groups was obtained by Sato and Kimura~\cite{SK} and Schpiz~\cite{Sh}. For more recent results on this subject, see~\cite{Ki} and references therein.

In this paper we study prehomogeneous modules of commutative linear algebraic groups. Denote by $\GG_m$ the multiplicative group and by $\GG_a$ the additive group of the ground field~$\KK$. It is well known that any connected commutative linear algebraic group $G$ is isomorphic to $(\GG_m)^{r}\times(\GG_a)^s$ with some non-negative integers $r$ and $s$, see \cite[Theorem~15.5]{Hum}. We say that $r$ is the \emph{rank} and $s$ is the \emph{corank} of the group $G$.

It is an impotant problem to describe regular actions of a commutative linear algebraic  group $G$ on algebraic varieties $X$ with an open orbit. If $s=0$ then $G$ is a torus and we come to the classical theory of toric varieties, see \cite{De,Oda,Fu,CLS}. The case $s=1$ is studied in~\cite{AK}. It turns out that the variety $X$ in this case is toric as well, and $G$-actions on $X$ with an open orbit are determined by Demazure roots of $X$.

Another extreme $r=0$ corresponds to embeddings of commutative unipotent (=vector) groups. This case is studied actively during last decades, see \cite{HT, Sha, Arz2, AS, Fe, FH, AR}.

The aim of this paper is to study linear $G$-actions with an open orbit and to give linearizability criteria for some classes of actions of commutative groups on affine spaces.
The~paper is organized as follows. In Section~\ref{sec2} we discuss preliminary results on prehomogeneous modules of commutative linear algebraic groups. Section~\ref{sec3} contains basic facts on finite dimensional algebras. We recall Hassett-Tschinkel correspondence between actions of commutative unipotent groups on projective spaces with an open orbit and local finite dimensional algebras. Also we list all $42$ local algebras of dimension up to $6$. The~classification is taken from~\cite[Section~2]{Ma}. It seems to be not widely known.

In Section~\ref{sec4} we show that every prehomogenenous module $V$ of a commutative linear algebraic group $G$ is isomorphic to the $G(A)$-module $A$, where $A$ is a finite dimensional commutative associative algebra $A$ with unit and $G(A)$ is its group of invertible elements (Theorem~\ref{tmain}). This result implies that the number of equivalence classes of prehomogeneous $G$-modules is finite if and only if the corank of the group $G$ is at most $5$ (Corollary~\ref{corpol}). We prove that the number of prehomogeneous modules with a finite number of orbits and the number of prehomogenenous modules such that the acting group is normalized by all invertible diagonal matrices are finite for every commutative linear algebraic group $G$ (Corollaries~\ref{cor1} and~\ref{cor2}). Such modules allow an explicit description in terms of the corresponding finite dimensional algebras. Proposition~\ref{propcyc} shows that every cyclic module of a commutative linear algebraic group is obtained from a prehomogeneous module of a bigger commutative group by restriction to an action of a subgroup.

In Section~\ref{sec5} we deal with prehomogeneous modules in the framework of the theory of affine algebraic monoids and group embeddings. Consider an affine algebraic monoid $S$ isomorphic as a variety to an affine space. Proposition~\ref{propmul} claims that the monoid $S$ is the multilplicative monoid of a finite dimensional algebra if and only if the action of the group $G(S)\times G(S)$ on $S$ by left and right multiplication is linearizable. It leeds to an alternative proof of Theorem~\ref{tmain}.

In the last section we consider additive actions on toric varieties $X$, i.e., regular actions $\GG_a^m\times X\to X$ with an open orbit. Lifting such actions to the spectrum of the Cox ring of the variety $X$ we obtain actions of a commutative linear algebraic group $G$ on $\AA^n$ with an open orbit. We show that such actions are linearizable if and only if $X$ is a big open toric subset of a product of projective spaces (Proposition~\ref{assact}).

In a forthcoming paper we plan to study non-linearizable actions of commutative linear algebraic groups on $\AA^n$ with an open orbit or, equivalently, commutative monoid structures on affine spaces that do not correspond to finite dimensional algebras.

\smallskip

\emph{Important Comment.} When the first version of this text has appeared, Professor Friedrich Knop pointed our attention to paper~\cite{KL}. Proposition~5.1 of~\cite{KL} is equivalent to our Theorem~\ref{tmain} and the result is proved there over an arbitrary field.

\section{Preliminaries}
\label{sec2}

Let $V$ be a finite dimensional vector space and $G$ a closed subgroup of the group $\GL(V)$. We say that $V$ is a \emph{prehomogeneous $G$-module} if the induced $G$-action on $V$ has an open orbit. A $G$-module $V$ is \emph{equivalent} to a $G'$-module $V'$ if there exists an isomorphism of vector spaces $V$ and $V'$ such that the induced isomorphism of the groups $\GL(V)$ and $\GL(V')$ identifies the subgroups $G$ and $G'$.

\begin{lemma} \label{lemdim}
If $V$ is a prehomogeneous module of a commutative linear algebraic group $G$, then $\dim V=\dim G$.
\end{lemma}

\begin{proof}
By definition, the module $V$ contains an open $G$-orbit $\OOO$. The stabilizer of a point on $\OOO$ acts trivialy on $\OOO$ and hence on $V$. Since the group $G$ acts on $V$ effectively, the action of $G$ on $\OOO$ is free, and $\dim V=\dim\OOO=\dim G$.
\end{proof}

\begin{remark} \label{polex}
For a commutative linear algebraic group $G$ there may exist a faithful $G$-module $V$ with $\dim V<\dim G$. For example, take $G=\GG_m\times(\GG_a)^{n^2}$ and its representation given by
$(2n\times 2n)$-matrices of the form
$$
\begin{pmatrix}
\lambda E & A \\
0 & \lambda E
\end{pmatrix}, \
\text{where} \ \lambda\in\GG_m \ \text{and} \ A\in\text{Mat}(n\times n,\KK).
$$
\end{remark}

\begin{lemma}\label{connected}
If a commutative linear algebraic group $G$ admits a prehomogeneous $G$-module~$V$, then $G$ is connected.
\end{lemma}

\begin{proof}
By definition, $G$ is isomorphic as a variety to a dense open subset of the module $V$. Hence $G$ is irreducible or, equivalently, connected.
\end{proof}

\begin{proposition}\label{norm}
Let $G$ be a commutative linear algebraic group and $V$ a prehomogeneous $G$-module. Then the group $G$ coincides with its centralizer in $\GL(V)$. In particular, $G$ is a maximal commutative subgroup of $\GL(V)$.
\end{proposition}

\begin{proof}
Let $C$ be the centralizer of $G$ in $\GL(V)$. The group $C$ preserves an open $G$-orbit $\OOO$ in $V$. The group of $G$-equivariant automorphisms of the orbit $\OOO$ coincides with $G$. Hence for every $c\in C$ there exists an element $g\in G$
whose action on $\OOO$ coincides with the action of $c$. Since $\OOO$ is open and dense in $V$, we conclude that $C=G$.
\end{proof}

\section{Finite dimensional algebras and Hassett-Tschinkel correspondence}
\label{sec3}

Let $A$ be a finite dimensional commutative associative algebra with unit over the ground field $\KK$. It is well known that $A$ admits a unique decomposition $A=A_1\oplus\ldots\oplus A_r$ into a direct sum of local algebras $A_i$ with maximal ideals $\mmm_i$; see,  e.g.,~\cite[Theorem~8.7]{AM}. Moreover, every algebra $A_i$ decomposes as a vector space to $\KK\oplus\mmm_i$, all elements in $\mmm_i$ are nilpotent and all elements in $A_i\setminus\mmm_i$ are invertible. In particular, the group of invertible elements $G(A_i)$ equals $\KK^{\times}\oplus\mmm_i$.

\smallskip

In~\cite{HT}, Hassett and Tschinkel  established a correspondence between local algebras $A$ of dimension $n$ and effective actions of the commutative unipotent group $\GG_a^{n-1}$ on the projective space $\PP^{n-1}$ with an open orbit.
Here the projective space is realized as the projectivization $\PP(A)$ and the $\GG_a^{n-1}$-action comes from the action of the group $\exp(\mmm)=1+\mmm$ on $A$ by multiplication; see also~\cite[Section~1]{AS}.

\smallskip

These results may be reformulated as follows, cf. \cite[Theorem~2.14]{HT} and \cite[Theorem~1.1]{AS}.

\begin{proposition} \label{propht}
Equivalence classes of prehomogeneous modules $V$ of the group $\GG_m\times\GG_a^{n-1}$, where the torus $\GG_m$ acts on $V$ by scalar multiplication, are in bijection with isomorphy classes of local algebras $A$ of dimension $n$. More precisely, every such prehomegneneous module is isomorphic to the $G(A)$-module $A$.
\end{proposition}

\begin{example}
Let us illustrate the bijection of Proposition~\ref{propht} for the local algebra $A=\KK[x_1,x_2]/(x_1x_2,x_1^3-x_2^3). $
Take a basis $\{1,x_1,x_2,x_1^2,x_2^2,x_1^3\}$ in $A$. The exponent
$$
\exp(\alpha_1x_1+\alpha_2x_2+\alpha_3x_1^2+\alpha_4x_2^2+\alpha_5x_1^3)
$$
is equal to
$$
1+\alpha_1x_1+\alpha_2x_2+\alpha_3x_1^2+\alpha_4x_2^2+\alpha_5x_1^3+\frac{1}{2}(\alpha_1^2x_1^2+\alpha_2^2x_2^2+2(\alpha_1\alpha_3+\alpha_2\alpha_4)x_1^3)+\frac{1}{6}(\alpha_1^3+\alpha_2^3)x_1^3.
$$
Multiplyng all basis vectors by this element, we obtain an explicit matrix form for the corresponding 6-dimensional prehomogeneous $(\GG_m\times\GG_a^5)$-module:

\smallskip

$$
\begin{pmatrix}
\lambda  & 0 & 0 & 0 & 0 & 0 \\
\lambda\alpha_1 & \lambda  & 0 & 0 & 0 & 0 \\
\lambda\alpha_2 & 0 & \lambda  & 0 & 0 & 0 \\
\lambda(\alpha_3+\frac{\alpha_1^2}{2}) & \lambda \alpha_1 & 0 & \lambda  & 0 & 0 \\
\lambda(\alpha_4+\frac{\alpha_2^2}{2}) & 0 & \lambda \alpha_2 & 0 & \lambda  & 0 \\
\lambda (\alpha_5+\alpha_1\alpha_3+\alpha_2\alpha_4+\frac{\alpha_1^3+\alpha_2^3}{6}) & \lambda (\alpha_3+\frac{\alpha_1^2}{2}) & \lambda(\alpha_4+\frac{\alpha_2^2}{2}) & \lambda\alpha_1 & \lambda\alpha_2 & \lambda
\end{pmatrix}
$$

\smallskip

\noindent with $\lambda\in\GG_m$ and  $\alpha_1,\ldots,\alpha_5\in\GG_a$.
\end{example}

The following result is explained in \cite[Section~3]{HT} using a classification of commuting nilpotent matrices from~\cite{ST}.

\begin{proposition} \label{le6}
The number of isomorphy classes of local algebras of dimension $n$ is finite if and only if $n\le6$.
\end{proposition}

\begin{proof}
A classification of local algebras of dimension $n\le 6$ can be extracted from~\cite[pages~136-150]{ST}. Explicitly it is given
in~\cite[Section~2]{Ma}. We reproduce this classification below.

For every $n\ge 7$ a continuous family of pairwise non-isomorphic local algebras of dimension $n$ is constructed in~\cite[Section~3]{HT};
see~\cite[Example~3.6]{HT} and the text preceding this example.
\end{proof}

In the following table we list all local algebras of dimension up to $6$. The classification is taken from~\cite[Section~2]{Ma}.

\begin{longtable}{|c|c|c|}
\hline
No. & $\dim A$ & Local algebra $A$ \\
\hline
1& $1$ & $\KK$ \\
\hline
2& $2$ & $\KK[x_1]/(x_1^2)$ \\
\hline
3& $3$ & $\KK[x_1]/(x_1^3)$ \\
\hline
4& $3$ & $\KK[x_1,x_2]/(x_1^2,x_2^2,x_1x_2)$ \\
\hline
5& $4$ & $\KK[x_1]/(x_1^4)$ \\
\hline
6& $4$ & $\KK[x_1,x_2]/(x_1^2,x_2^2)$ \\
\hline
7& $4$ & $\KK[x_1,x_2]/(x_1^3,x_1x_2,x_2^2)$ \\
\hline
8& $4$ & $\KK[x_1,x_2,x_3]/(x_i^2,x_ix_j)$ \\
\hline
9& $5$ & $\KK[x_1]/(x_1^5)$ \\
\hline
10& $5$ & $\KK[x_1,x_2]/(x_1x_2,x_1^3-x_2^2)$ \\
\hline
11& $5$ & $\KK[x_1,x_2]/(x_1^3,x_2^3,x_1x_2)$ \\
\hline
12& $5$ & $\KK[x_1,x_2]/(x_1^4,x_2^2,x_1x_2)$ \\
\hline
13& $5$ & $\KK[x_1,x_2]/(x_1^3,x_2^2,x_1^2x_2)$ \\
\hline
14& $5$ & $\KK[x_1,x_2,x_3]/(x_1x_2,x_1x_3,x_2x_3,x_1^2-x_2^2,x_1^2-x_3^2)$ \\
\hline
15& $5$ & $\KK[x_1,x_2,x_3]/(x_1^2,x_1x_2,x_1x_3,x_2x_3,x_2^2-x_3^2)$ \\
\hline
16& $5$ & $\KK[x_1,x_2,x_3]/(x_1^3,x_2^2,x_3^2,x_1x_2,x_1x_3,x_2x_3)$ \\
\hline
17& $5$ & $\KK[x_1,x_2,x_3,x_4]/(x_i^2,x_ix_j)$ \\
\hline
18& $6$ & $\KK[x_1]/(x_1^6)$ \\
\hline
19& $6$ & $\KK[x_1,x_2]/(x_1x_2,x_1^4-x_2^2)$ \\
\hline
20& $6$ & $\KK[x_1,x_2]/(x_1x_2,x_1^3-x_2^3)$ \\
\hline
21& $6$ & $\KK[x_1,x_2]/(x_1^3,x_2^2)$ \\
\hline
22& $6$ & $\KK[x_1,x_2]/(x_1^5,x_1x_2,x_2^2)$ \\
\hline
23& $6$ & $\KK[x_1,x_2]/(x_1^4,x_1x_2,x_2^3)$ \\
\hline
24& $6$ & $\KK[x_1,x_2]/(x_1^3,x_1^2x_2,x_1x_2^2,x_2^3)$ \\
\hline
25& $6$ & $\KK[x_1,x_2]/(x_1^4,x_1^2x_2,x_1^3-x_2^2)$ \\
\hline
26& $6$ & $\KK[x_1,x_2]/(x_1^4,x_1^2x_2,x_2^2)$ \\
\hline
27& $6$ & $\KK[x_1,x_2,x_3]/(x_1^2,x_2^2,x_3^2,x_1x_2-x_1x_3)$ \\
\hline
28& $6$ & $\KK[x_1,x_2,x_3]/(x_2^2,x_3^2,x_1x_2,x_1^2-x_2x_3)$ \\
\hline
29& $6$ & $\KK[x_1,x_2,x_3]/(x_1^2,x_2^2,x_3^2,x_2x_3)$ \\
\hline
30& $6$ & $\KK[x_1,x_2,x_3]/(x_1^2,x_2^2,x_1x_3,x_2x_3,x_1x_2-x_3^3)$ \\
\hline
31& $6$ & $\KK[x_1,x_2,x_3]/(x_1^2-x_3^3,x_2^2,x_1x_2,x_1x_3,x_2x_3)$ \\
\hline
32& $6$ & $\KK[x_1,x_2,x_3]/(x_1^3,x_2^2,x_3^2,x_1x_2,x_1x_3)$ \\
\hline
33& $6$ & $\KK[x_1,x_2,x_3]/(x_1^2,x_2^2,x_3^2,x_1x_2-x_1x_3-x_2x_3)$ \\
\hline
34& $6$ & $\KK[x_1,x_2,x_3]/(x_1^3,x_2^2,x_1x_3,x_2x_3,x_1x_2-x_3^2)$ \\
\hline
35& $6$ & $\KK[x_1,x_2,x_3]/(x_1^4,x_2^2,x_3^2,x_1x_2,x_1x_3,x_2x_3)$ \\
\hline
36& $6$ & $\KK[x_1,x_2,x_3]/(x_1^3,x_2^3,x_3^2,x_1x_2,x_1x_3,x_2x_3)$ \\
\hline
37& $6$ & $\KK[x_1,x_2,x_3]/(x_1^3,x_2^2,x_3^2,x_1^2x_2,x_1x_3,x_2x_3)$ \\
\hline
38& $6$ & $\KK[x_1,x_2,x_3,x_4]/(x_i^2,x_1x_2,x_1x_3,x_2x_4,x_3x_4,x_1x_4-x_2x_3)$ \\
\hline
39& $6$ & $\KK[x_1,x_2,x_3,x_4]/(x_1^2,x_2^2,x_4^2,x_1x_3,x_1x_4,x_2x_3,x_2x_4,x_1x_2-x_3^2)$ \\
\hline
40& $6$ & $\KK[x_1,x_2,x_3,x_4]/(x_i^2,x_1x_3,x_1x_4,x_2x_3,x_2x_4,x_3x_4)$ \\
\hline
41& $6$ & $\KK[x_1,x_2,x_3,x_4]/(x_1^3,x_2^2,x_3^2,x_4^2,x_ix_j,i\ne j)$ \\
\hline
42& $6$ & $\KK[x_1,x_2,x_3,x_4,x_5]/(x_i^2,x_ix_j)$ \\
\hline
\end{longtable}

\section{Prehomogeneous and cyclic modules}
\label{sec4}

In this section we show that every prehomogeneous module $V$ of a commutative linear algebraic group $G$ comes from a finite dimensional  commutative associative algebra $A$ with unit and discuss some corollaries of this result.

\begin{theorem} \label{tmain}
Let $G$ be a commutative linear algebraic group and $V$ a prehomogeneous $G$-module. Then there exists a finite dimensional commutative associative algebra $A$ with unit such that the $G$-module $V$ is isomorphic to the $G(A)$-module $A$. Moreover,
two prehomogeneous modules are equivalent if and only if the corresponding algebras are isomorphic.
\end{theorem}

\begin{proof}
By Lemma~\ref{connected}, the group $G$ is connected. Hence $G$ is isomorphic to $T\times\GG_a^m$. Any orbit of an action a unipotent group on an affine variety is closed, see, e.g., \cite[Section~1.3]{PV}. Hence the torus T in $G$ has positive dimension.

Denote by $\XX(T)$ the lattice of characters of the torus $T$. Consider the weight decomposition of the module $V$ with respect to the torus $T$:
$$
V=\bigoplus_{\chi\in\XX(T)} V_{\chi}, \quad \text{where} \quad V_{\chi}=\{v\in V \ | \ tv=\chi(t)v \ \text{for\, all} \ t\in T\}.
$$
Each subspace $V_{\chi}$ is invariant under the group $G$ and we have induced representations $\rho_i\colon G\to\GL(V_{\chi_i})$ for all weights $\chi_1,\ldots\chi_r$ with nonzero weight subspaces. Denote the image $\rho_i(G)$ by $G_i$
and let $V_i:=V_{\chi_i}$. Then $G$ is contained in $G_1\times\ldots\times G_r$, and thus the componentwise action of  $G_1\times\ldots\times G_r$  on $V$ has an open orbit. This implies that each $G_i$ acts on $V_i$ with an open orbit.

By Proposition~\ref{propht}, there exist local algebras $A_i$ such that the $G_i$-modules $V_i$ are isomorphic to the $G(A_i)$-modules $A_i$ for all $i=1,\ldots,r$. In particular, we have $\dim G_i=\dim V_i$. By Lemma~\ref{lemdim}, the group $G$ coincides with $G_1\times\ldots\times G_r$, and thus the $G$-module $V$ is isomorphic to the $G(A)$-module $A$ for $A=A_1\oplus\ldots\oplus A_r$.

The last assertion follows from Proposition~\ref{propht} and uniqueness of decomposition of an algebra $A$ into local summands.
\end{proof}

\begin{corollary} \label{corpol}
Let $G$ be a commutative linear algebraic group. The number of isomorphy classes of prehomogeneous $G$-modules is finite if and only if the corank of $G$ is at most 5.
\end{corollary}

\begin{proof}
By Proposition~\ref{le6}, the number of isomorphy classes of finite dimensional algebras $A$ with fixed dimensions of local summands is finite if and only if the dimensions of local summands do not exeed $6$. In our case the dimensions of local summands
do not exceed the corank of $G$ plus $1$, and this value is achieved when all local summands of $A$ except one are $\KK$.
\end{proof}

\begin{remark}
Recall that two elements $a$ and $b$ of an algebra $A$ are \emph{associated} if there exists an element $c\in G(A)$ such that $a=bc$. Theorem~\ref{tmain} implies that $G$-orbits in a prehomogeneous $G$-module $V$ are in bijection with
association classes in the corresponding algebra $A$.
\end{remark}

For positive integers $n$ and $r$, we denote by $p_r(n)$ the number of partitions $n=n_1+\ldots+n_r$ with $n_1\ge\ldots\ge n_r\ge 1$.

\begin{corollary} \label{cor1}
Let $G$ be a commutative linear algebraic group of dimension $n$ and rank~$r$. Then there exist precisely $p_r(n)$ prehomogeneous $G$-modules with a finite number of $G$-orbits. The corresponding algebras $A$ are precisely the algebras
of the form $\KK[x]/(f(x))$, where $f(x$) is a polynomial of degree $n$ with $r$ distinct roots.
\end{corollary}

\begin{proof}
The number of $G(A)$-orbits in $A$ is finite if and only if the number of $G(A_i)$-orbits in $A_i$ is finite for every local summand $A_1,\ldots,A_r$ in $A$. By~\cite[Proposition~3.7]{HT}, the number of $G(A_i)$-orbits in $A_i$ is finite
if and only if the algebra $A_i$ is isomorphic to $\KK[x]/(x^{n_i})$. It shows that the algebra $A$ is uniquely determined by dimensions $n_1,\ldots,n_r$ of its local summands.
\end{proof}

\begin{remark}
A classification of irreducible prehomogeneous modules with finitely many orbits is obtained in~\cite[Theorem~8]{SK}.
\end{remark}

\begin{corollary} \label{cor2}
Let $G$ be a commutative linear algebraic group of dimension $n$ and rank~$r$. Then there exist precisely $p_r(n)$ prehomogeneous $G$-modules $V$ such that the group $G$ is normalized by the group of all invertible diagonal matrices on $V$.  The corresponding algebras $A$ are precisely the algebras with $a^2=0$ for every nilpotent $a\in A$.
\end{corollary}

\begin{proof}
Clearly, the action of $G(A)$ on $A$ is normalized by the group of diagonal matrices if and only if it holds for every local summand of $A$. Thus it suffices to prove that in every dimension there exists a unique local algebra $A$ such that
the action of $G(A)$ on $A$ is normilized by all diagonal matrices.

In~\cite{AR}, actions with an open orbit of the group $\GG_a^m$ on toric varieties are studied. It is shown in~\cite[Theorem~3.6]{AR} that any two such actions normalized by the maximal torus are isomorphic. In the case of the projective space $\PP^m$ this unique isomorphy class of actions corresponds to the local algebra
$$
\KK[x_1,\ldots,x_{m-1}]/(x_ix_j, \ 1\le i\le j\le m-1),
$$
or, equivalently, $A=\KK\oplus\mmm$ with $\mmm^2=0$; see~\cite[Example~6.1]{AR} and \cite[Proposition~2.7]{AS}. This completes the proof.
\end{proof}

Let us recall that a regular function $f$ on a $G$-module $V$ is a semi-invariant if $f(gv)=\mu(g)f(v)$ for some character $\mu$ of the group $G$ and all $g\in G$ and $v\in V$.

\begin{corollary}
Let $G$ be a commutative linear algebraic group of rank $r$ and $V$ a prehomogeneous $G$-module.Then the complement of an open $G$-orbit in $V$ is a union of hyperplanes $H_1,\ldots,H_r$. Moreover, every semi-invariant on $V$ is a monomial in the linear functions defining the hyperplanes $H_1,\ldots,H_r$.
\end{corollary}

\begin{proof}
The first assertion follows from the decomposition $A_i=\KK\oplus\mmm_i$ and the equality $G(A_i)=A_i\setminus\mmm_i$  for every local summand $A_i$ in $A$. For the second assertion we observe that the support of the divisor of zeroes of a semi-invariant $f$ on $V$ is contained in the complement of the open $G$-orbit, i.e., in the  union of the hyperplanes $H_1,\ldots,H_r$. This observation implies the claim.
\end{proof}

The next proposition shows that finite dimensional algebras can be used to study much wider class of modules than the class of prehomogenenous modules. Let us recall that a vector $v$ in a $G$-module $V$ is \emph{cyclic} if the linear span of the orbit $Gv$ coincides with $V$. A~$G$-module $V$ is \emph{cyclic} if it has a cyclic vector.

\begin{proposition} \label{propcyc}
Let $V$ be a cyclic module of a commutative linear algebraic group $G$. Then there exist a finite dimensional commutative associative algebra $A$ with unit and an injective homomorphism $G\to G(A)$ such that the $G$-modules $V$ and $A$ are isomorphic.
\end{proposition}

\begin{proof}
It is well known that every commutative linear algebraic group $G$ is isomorphic to a direct product $L\times T\times\GG_a^m$, where $L$ is a finite abelian group and $T$ is a torus. Moreover, the action of the group $L\times T$ on $V$ is diagonalizable.
Consider the weight decomposition $V=\oplus V_{\nu}$ of the module $V$ with respect to $L\times T$. Then every subspace $V_{\nu}$ is $G$-invariant and hence is a cyclic $G$-module. Enlarging the group $G$ we assume that $T$ consists of all
invertible operators which act on every $V_{\nu}$ by scalar multiplication.

It suffices to show that each $V_{\nu}$ may be identified with some local algebra $A$ in such a way that the action of $G$ on $V_{\nu}$ is a restriction of the action of $G(A)$ on $A$ by multiplication to a closed subgroup in $G(A)$. It follows from
\cite[Theorem~2.14]{HT}, see \cite[Remark~1.2]{AS} for details.
\end{proof}

\begin{remark}
The condition that the $G$-module $V$ is cyclic is essential, see Remark~\ref{polex}.
\end{remark}


\section{Affine monoids and group embeddings}
\label{sec5}

An affine algebraic monoid is an irreducible affine variety $S$ with an associative multiplication
$$
\mu \colon S\times S\to S,\quad (a,b)\mapsto ab,
$$
that is a morphism of algebraic varieties, and a unit element $e\in S$ such that $ea=ae=a$ for all $a\in S$. An example of an affine algebraic monoid is the multiplicative monoid of a finite dimensional associative algebra with unit.

The group of invertible elements $G(S)$ of a monoid $S$ is open in $S$. Moreover, $G(S)$ is a linear algebraic group. For a general theory of affine algebraic monoids, we refer to~\cite{Vi,Ri,Re}.

By a \emph{group embedding} we mean an irreducible affine variety $X$ with an open embedding $G\hookrightarrow X$ of a linear algebraic group $G$ such that both actions by left and right multiplications of $G$ on itsefl can be extended to $G$-actions on $X$. In other words, the variety $X$ is a $(G\times G)$-equivariant open embedding of the homogeneous space $(G\times G)/\Delta(G)$, where $\Delta(G)$ is the diagonal in $G\times G$.

Any affine monoid $S$ defines a group embedding $G(S)\hookrightarrow S$. The converse statement is proved in~\cite{Vi} under the assumption that the group $G$ is reductive and in~\cite{Ri} for arbitrary~$G$. For convenience of the reader we reproduce below the proof from~\cite{Ri}.

\begin{proposition} \label{proprit}
Let $G$ be a linear algebraic group. Then for every group embedding $G\hookrightarrow S$ there exists a structure of an affine algebraic monoid on $S$ such that the group $G$ coincides with the group of invertible elements $G(S)$.
\end{proposition}

\begin{proof}
Let us prove that the multiplication morphism $G\times G\to G$ can be extended to a morphism $S\times S\to S$. Consider the morphisms given by left and right multiplication
$$
G\times S\to S \quad \text{and} \quad S\times G\to S
$$
and the corresponding comorphisms
$$
\KK[S]\to \KK[G]\otimes\KK[S] \quad \text{and} \quad \KK[S]\to \KK[S]\times\KK[G].
$$
Since the morphisms $G\times S\to S$ and $S\times G\to S$ extend the multiplication $G\times G\to G$, the image of the subalgebra $\KK[S]$ of $\KK[G]$ is contained in the intersection
$$
 (\KK[G]\otimes\KK[S])\cap(\KK[S]\otimes\KK[G])=\KK[S]\otimes\KK[S].
$$
This provides the desired extended morphism $S\times S\to S$. Such a morphism has the associativity property because it holds on an open dense subset $G$ in $S$. Similarly, the unit element $e\in G$ satisfies the property $es=se=s$ for all $s\in S$.

Every element of the group $G$ is invertible in $S$. Conversily, if $s\in S$ is invertible then the subvariety $sG$ is open in $S$ and thus the intersection $G\cap sG$ is non-empty. We conclude that $s$ lies in $G$ and $G(S)=G$.
\end{proof}

Let us recall that an action $G\times\AA^n\to\AA^n$ of a linear algebraic group $G$ is \emph{linearizable}, if the image of $G$ in $\Aut(X)$ is conjugate to a subgroup of the group $\GL_n(\KK)$ of all linear transformations of $\AA^n$.

\begin{proposition} \label{propmul}
Let $S$ be an affine algebraic monoid. Assume that the variety $S$ is isomorphic to an affine space. Then $S$ is the multiplicative monoid of a finite dimensional algebra
if and only if the action of the group $G(S)\times G(S)$ on $S$ by left and right multiplication is linearizable.
\end{proposition}

\begin{proof}
Assume that the group $G(S)\times G(S)$ acts linearly on the vector space $V$ identified with the variety $S$. The multiplication $V\times V\to V$ is given by the comorphism $\KK[V]\to\KK[V]\otimes\KK[V]$. Since the $(G(S)\times G(S))$-action on $V$ is linear,
for the restriction of the comorphism to the subspace $V^*\subseteq\KK[V]$ of all linear functions on $V$ we have
$$
V^*\to \KK[G]\otimes V^* \quad \text{and} \quad V^*\to V^*\otimes\KK[G].
$$
So the image of $V^*$ is contained in the intersection $(\KK[G]\otimes V^*)\cap(V^*\otimes\KK[G])=V^*\otimes V^*$. Hence the multiplication on $V$ is given by the  linear map $V\otimes V\to V$ dual to $V^*\to V^*\otimes V^*$. This proves that the multiplication on $V$ is bilinear and thus the monoid $S$ is isomorphic to the multiplicative monoid of a finite dimensional algebra.

The converse implication is straightforward.
\end{proof}

Let us recall that an affine algebraic monoid $S$ is \emph{reductive} if the group $G(S)$ is a reductive linear algebraic group.
Since every action of a reductvie group on an affine space with an open orbit is linearizable~\cite{Lu}, we obtain the following result.

\begin{corollary} \label{corred}
If $S$ is a reductive monoid and the variety $S$ is isomorphic to an affine space, then $S$ is the multiplicative monoid of a finite dimensional algebra.
\end{corollary}

Clearly, the multiplicative monoid of a finite dimensional algebra $R$ is reductive if and only if $R$ is a semisimple $R$-module. By the Artin-Wedderburn Theorem,
this is the case if and only if $R$ is a direct sum of matrix algebras $\text{Mat}(n_i\times n_i,\KK)$.

\begin{remark}
For commutative monoids the statement of Corollary~\ref{corred} does not hold.
\end{remark}

Propositions~\ref{proprit} and~\ref{propmul} provide an alternative proof of Theorem~\ref{tmain}. Indeed, for a commutative linear algebraic group $G$ and a prehomogeneous $G$-module $V$ the orbit map to an open $G$-orbit gives rise to a group embedding $G\hookrightarrow V$. Since the $G$-action on $V$ is linear, the monoid structure on $V$ comes from a finite dimensional algebra $A$ with the underlying vector space $V$ and the group $G$ is identified with the group of invertible elements $G(A)$. The group $G(A)$ is commutative and dense in $A$, so the algebra $A$ is commutative as well. Since the multiplication on $V$ is defined by the comorphism $\KK[V]\to\KK[G]\otimes\KK[V]$, the algebra structure on $V$ is uniquely determined by the $G$-module structure on $V$.

\begin{example}
Let $V$ be a prehomogeneous $G$-module with trivial generic stabilizer, where $G$ is a non-commutative linear algebraic group. The inclusion of an open orbit $G\to V$ need not be a group embedding. For instance, take the group
$$
G=\left\{
\begin{pmatrix}
t & a \\
0 & t^{-1}
\end{pmatrix}, \
t\in\GG_m, \ a\in\GG_a\right\}
$$
and its tautological module $\KK^2$. The orbit  of the vector $(0,1)$ is open in $\KK^2$, it consists of the vectors $(a,t^{-1})$ or, equivalently, of the vectors $(x,y)$, $y\ne 0$. The right multiplication by an element
$$
\begin{pmatrix}
s & b \\
0 & s^{-1}
\end{pmatrix}^{-1}
$$
gives the vector $(sa-tb,t^{-1}s)$ or, equivalently, $(sx-by^{-1},sy)$. Such an action can not be extended to $\KK^2$.
\end{example}


\section{Additive actions on toric varieties and Cox rings}
\label{sec6}

Let $X$ be an irreducible algebraic variety over the ground field $\KK$. An \emph{additive action} on $X$ is a regular faithful action $\GG_a^m\times X\to X$ with an open orbit. Let us recall that a variety $X$ is \emph{toric} if $X$ is normal
and there exists an action of an algebraic torus $T$ on $X$ with an open orbit. Additive actions on toric varieties are studied in~\cite{AR}.

If a variety $X$ admits an additive action, then every regular invertible function on $X$ is constant and the divisor class group $\Cl(X)$ is a free finitely generated abelian group~\cite[Lemma~1]{APS}. For a toric variety $X$ these conditions imply that $X$ can be realized as a good quotient $\pi\colon U\stackrel{/\!/H}\longrightarrow X$ of an open subset $U\subseteq\AA^n$ whose complement is a collection of coordinate subspaces of codimensions at least $2$ in $\AA^n$  by a linear action of a torus $H$. Such a realization can be chosen in a canonical way. Namely, the Cox ring
$$
R(X)=\bigoplus_{[D]\in\Cl(X)} H^0(X,D)
$$
of a toric variety $X$ is a polynomial ring graded by the group $\Cl(X)$. The grading defines a linear action of the characteristic torus $H:=\Spec(\KK[\Cl(X)])$ on the total coordinate space $\AA^n:=\Spec(R(X))$. A canonically defined open subset
$U\subseteq\AA^n$, whose complement is a union of some coordinate subspaces of codimensions at least $2$, gives rise to the so-called characteristic space  $p\colon U\stackrel{/\!/H}\longrightarrow X$; we refer to ~\cite{Cox} and \cite[Chapter~II]{ADHL}
for details.

An additive action $\GG_a^m\times X\to X$ can be lifted to an action $\GG_a^m\times\AA^n\to\AA^n$ on the total coordinate space commuting with the $H$-action. This defines an action $G\times\AA^n\to\AA^n$ of the commutative group $G:=H\times\GG_a^m$ with an open orbit. Let us say that the action $G\times\AA^n\to\AA^n$ is \emph{associated} with the given additive action on a toric variety $X$.

We say that a toric variety $X$ is a \emph{big open subset} of a toric variety $X'$ if $X$ is isomorphic to an open toric subset $W$ of the variety $X'$ such that $\codim_{X'} X'\setminus W\ge 2$.

\begin{proposition} \label{assact}
An action $G\times\AA^n\to\AA^n$ associated with an additive action on a toric variety $X$ is linearizable if and only if $X$ is a big open subset of a product of projective spaces.
\end{proposition}

\begin{proof}
It follows from Theorem~\ref{tmain} that if an action $G\times\AA^n\to\AA^n$ is linearizable, then in suitable coordinates we have $\AA^n=V_1\oplus\ldots\oplus V_r$, where $r=\rk(G)=\dim H$ and every element $(t_1,\ldots,t_r)\in H$ acts on every subspace $V_i$ via scalar multiplication by $t_i$. Let $d_i:=\dim V_i$. The torus $H$ acts on $\AA^n$ linearly with characters $e_1 (d_1 \text{times}),\ldots,e_r (d_r \text{times})$, where $e_1,\ldots,e_r$ form a basis of the lattice of characters $\XX(H)$.

It is easy to show (see, e.g.,~\cite[Exercise~2.13]{ADHL}) that there is a unique maximal open subset $U$ in $\AA^n$ such that there exists a good quotient $\pi\colon U\stackrel{/\!/H}\longrightarrow X$ which is the characteristic space of $X$; namely, $U=(V_1\setminus\{0\})\times\ldots\times(V_r\setminus\{0\})$ and $X=\PP(V_1)\times\ldots\times\PP(V_r)$. Other open subsets with this property are contained in $U$ and correspond to big open toric subsets of $\PP(V_1)\times\ldots\times\PP(V_r)$.

Conversely, consider an additive action $\GG_a^n\times X\to X$ on a big open toric subset $X$ of $\PP(V_1)\times\ldots\times\PP(V_r)$. The Picard group of $X$ is freely generated by the line bundles $L_1,\ldots,L_r$ corresponding
to ample generators of the Picard groups of the factors $\PP(V_1),\ldots,\PP(V_r)$. The space of global sections of $L_i$ is identified with the dual space $V_i^*$. By~\cite[Section~2.4]{KKLV}, every line bundle $L_i$ admits a $\GG_a^m$-linearization,
and thus the lifted action of the group $G=H\times\GG_a^m$ to the total coordinate space $\AA^n=V_1\oplus\ldots\oplus V_r$ of $X$ is linear.
\end{proof}

\begin{example}
Consider the action $\GG_a^n\times\AA^n\to\AA^n$ by translations. This is an additive action on a toric variety, and the associated action coincides with the original one. Since the action is transitive, it has no fixed point and thus it is not linearizable.
\end{example}

\begin{example}
Let $X$ be the Hirzebruch surface $\FF_d$. This toric variety admits an additive action normalized by the acting torus. The lifting of this action to the Cox ring extends to an action of the group $G=\GG_m^2\times\GG_a^2$ on $\AA^4$ with an open orbit. Explicitly this action is given by
$$
(x_1,x_2,x_3,x_4)\mapsto (\lambda_1x_1,\lambda_2x_2.\lambda_1x_3+\lambda_1\alpha_1x_1,\lambda_1^d\lambda_2x_4+\lambda_1^d\lambda_2\alpha_2x_1^dx_2), \  \lambda_1,\lambda_2\in\GG_m, \alpha_1,\alpha_2\in\GG_a.
$$
see \cite[Example~6.4]{AR}. By~Proposition~\ref{assact}, this action is not linearizable for $d\ge 1$. If $d=0$ then $X\cong\PP^1\times \PP^1$, and the action is linear.

\end{example}


\end{document}